\documentclass[12pt,a4paper]{article}%
\usepackage[utf8]{inputenc}
\usepackage{amsmath}
\usepackage{amsfonts}
\usepackage{amssymb}
\usepackage{graphicx}
\usepackage[space]{grffile}%
\providecommand{\U}[1]{\protect\rule{.1in}{.1in}}
\newtheorem{theorem}{Theorem}
\newtheorem{theorem*}{Example}

\newtheorem{conjecture}[theorem]{Conjecture}
\newtheorem{corollary}[theorem]{Corollary}

\newtheorem{definition}[theorem]{Definition}
\newtheorem{example}[theorem*]{Example}

\newtheorem{lemma}[theorem]{Lemma}

\newtheorem{proposition}[theorem]{Proposition}
\newtheorem{remark}[theorem]{Observation}

\newenvironment{proof}[1][Proof]{\noindent\textbf{#1.} }{\ \hfill \rule{0.5em}{0.5em}\bigskip}
\graphicspath{{D:/Dropbox/Riste-Sedlar/MetricDim/Paper 5/Slike/}}

\begin{document}

\title{Vertex and edge metric dimensions of cacti}
\author{Jelena Sedlar$^{1,3}$,\\Riste \v Skrekovski$^{2,3}$ \\[0.3cm] {\small $^{1}$ \textit{University of Split, Faculty of civil
engineering, architecture and geodesy, Croatia}}\\[0.1cm] {\small $^{2}$ \textit{University of Ljubljana, FMF, 1000 Ljubljana,
Slovenia }}\\[0.1cm] {\small $^{3}$ \textit{Faculty of Information Studies, 8000 Novo
Mesto, Slovenia }}\\[0.1cm] }
\maketitle

\begin{abstract}
In a graph $G,$ a vertex (resp. an edge) metric generator is a set of vertices
$S$ such that any pair of vertices (resp. edges) from $G$ is distinguished by
at least one vertex from $S.$ The cardinality of a smallest vertex (resp.
edge) metric generator is the vertex (resp. edge) metric dimension of $G.$ In
\cite{SedSkreUnicyclic} we determined the vertex (resp. edge) metric dimension
of unicyclic graphs and that it takes its value from two consecutive integers.
Therein, several cycle configurations were introduced and the vertex (resp.
edge) metric dimension takes the greater of the two consecutive values only if
any of these configurations is present in the graph. In this paper we extend
the result to cactus graphs i.e. graphs in which all cycles are pairwise edge
disjoint. We do so by defining a unicyclic subgraph of $G$ for every cycle of
$G$ and applying the already introduced approach for unicyclic graphs which
involves the configurations. The obtained results enable us to prove the cycle
rank conjecture for cacti. They also yield a simple upper bound on metric
dimensions of cactus graphs and we conclude the paper by conjecturing that the
same upper bound holds in general.

\end{abstract}



\section{Introduction}

The concept of metric dimension was first studied in the context of navigation
system in various graphical networks \cite{HararyVertex}. There the robot
moves from one vertex of the network to another, and some of the vertices are
considered to be a landmark which helps a robot to establish its position in a
network. Then the problem of establishing the smallest set of landmarks in a
network becomes a problem of determining a smallest metric generator in a
graph \cite{KhullerVertex}.

Another interesting application is in chemistry where the structure of a
chemical compound is frequently viewed as a set of functional groups arrayed
on a substructure. This can be modeled as a labeled graph where the vertex and
edge labels specify the atom and bond types, respectively, and the functional
groups and substructure are simply subgraphs of the labeled graph
representation. Determining the pharmacological activities related to the
feature of compounds relies on the investigation of the same functional groups
for two different compounds at the same point \cite{ChartrandVertex}. Various
other aspects of the notion were studied \cite{BuczkowskiVertex, FehrVertex,
KleinVertex, MelterVertex} and a lot of research was dedicated to the
behaviour of metric dimension with respect to various graph operations
\cite{CaceresVertex, ChartrandVertex, SaputroVertex, YeroCoronaVertex}.

In this paper, we consider only simple and connected graphs. By $d(u,v)$ we
denote the distance between a pair of vertices $u$ and $v$ in a graph $G$. A
vertex $s$ from $G$ \emph{distinguishes} or \emph{resolves} a pair of vertices
$u$ and $v$ from $G$ if $d(s,u)\not =d(s,v).$ We say that a set of vertices
$S\subseteq V(G)$ is a \emph{vertex metric generator,} if every pair of
vertices in $G$ is distinguished by at least one vertex from $S.$ The
\emph{vertex metric dimension} of $G,$ denoted by $\mathrm{dim}(G),$ is the
cardinality of a smallest vertex generator in $G$. This variant of metric
dimension, as it was introduced first, is sometimes called only metric
dimension and the prefix "vertex" is omitted.

In \cite{TratnikEdge} it was noticed that there are graphs in which none of
the smallest metric generators distinguishes all pairs of edges, so this was
the motivation to introduce the notion of the edge metric generator and
dimension, particularly to study the relation between $\mathrm{dim}(G)$ and
$\mathrm{edim}(G).$

The distance $d(u,vw)$ between a vertex $u$ and an edge $vw$ in a graph $G$ is
defined by $d(u,vw)=\min\{d(u,v),d(u,w)\}.$ Recently, two more variants of
metric dimension were introduced, namely the edge metric dimension and the
mixed metric dimension of a graph $G.$ Similarly as above, a vertex $s\in
V(G)$ \emph{distinguishes} two edges $e,f\in E(G)$ if $d(s,e)\neq d(s,f).$ So,
a set $S\subseteq V(G)$ is an \emph{edge metric generator} if every pair of
vertices is distinguished by at least one vertex from $S,$ and the cardinality
of a smallest such set is called the \emph{edge metric dimension} and denoted
by $\mathrm{edim}(G).$ Finally, a set $S\subseteq V(G)$ is a \emph{mixed
metric generator} if it distinguishes all pairs from $V(G)\cup E(G),$ and the
\emph{mixed metric dimension}, denoted by $\mathrm{mdim}(G)$, is defined as
the cardinality of a smallest such set in $G$.

This new variant also attracted a lot of attention \cite{GenesonEdge,
HuangApproximationEdge, PeterinEdge, ZhangGaoEdge, ZhuEdge, ZubrilinaEdge},
with one particular direction of research being the study of unicyclic graphs
and the relation of the two dimensions on them \cite{Knor, SedSkreBounds,
SedSkreUnicyclic}. The mixed metric dimension is then a natural next step, as
it unifies these two concepts. It was introduced in \cite{KelencMixed} and
further studied in \cite{SedSkrekMixed, SedSkreTheta}. A wider and systematic
introduction to these three variants of metric dimension can be found in
\cite{KelPhD}.

In this paper we establish the vertex and the edge metric dimension of cactus
graph, using the approach from \cite{SedSkreUnicyclic} where the two
dimensions were established for unicyclic graphs. The extension is not
straightforward, as in cactus graphs a problem with indistinguishable pairs of
edges and vertices may arise from connecting two cycles, so additional
condition will have to be introduced.

\section{Preliminaries}

A \emph{cactus} graph is any graph in which all cycles are pairwise edge
disjoint. Let $G$ be a cactus graph with cycles $C_{1},\ldots,C_{c}$ and let
$g_{i}$ denote the length of a cycle $C_{i}$ in $G.$ For a vertex $v$ of a
cycle $C_{i},$ denote by $T_{v}(C_{i})$ the connected component of
$G-E(C_{i})$ which contains $v.$ If $G$ is a unicyclic graph, then
$T_{v}(C_{i})$ is a tree, otherwise $T_{v}(C_{i})$ may contain a cycle. When
no confusion arises from that, we will use the abbreviated notation $T_{v}.$ A
\emph{thread} hanging at a vertex $v\in V(G)$ of degree $\geq3$ is any path
$u_{1}u_{2}\cdots u_{k}$ such that $u_{1}$ is a leaf, $u_{2},\ldots,u_{k}$ are
of degree $2,$ and $u_{k}$ is connected to $v$ by an edge. The number of
threads hanging at $v$ is denoted by $\ell(v).$

We say that a vertex $v\in V(C_{i})$ is \emph{branch-active} if $\deg(v)\geq4$
or $T_{v}$ contains a vertex of degree $\geq3$ distinct from $v$. We denote
the number of branch-active vertices on $C_{i}$ by $b(C_{i}).$ If a vertex $v$
from a cycle $C_{i}$ is branch-active, then $T_{v}$ contains both a pair of
vertices and a pair of edges which are not distinguished by any vertex outside
$T_{v}$, see Figure \ref{Fig_branching}.

\begin{figure}[h]
\begin{center}
\includegraphics[scale=0.9]{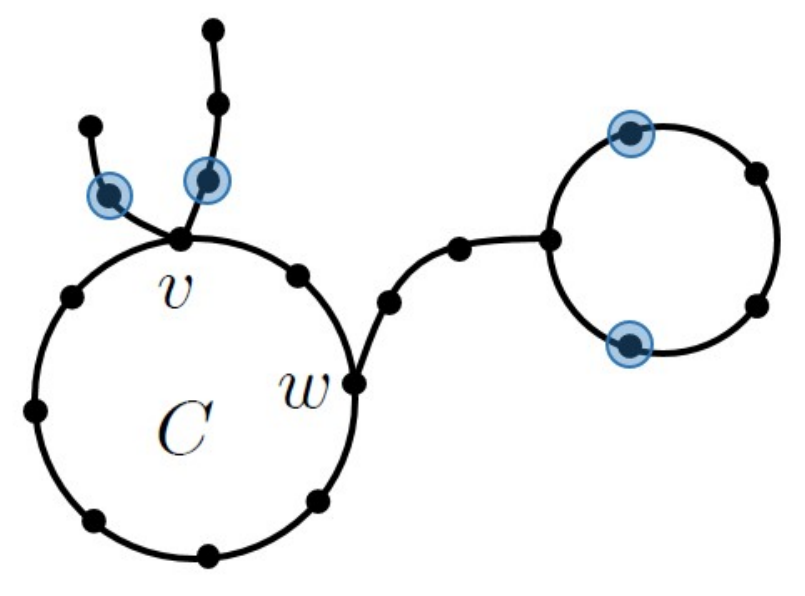}
\end{center}
\caption{A cactus graph with two cycles. On the cycle $C$ vertices $v$ and $w$
are branch-active, and a pair of vertices is marked in $T_{v}$ and $T_{w}$
which is not distinguished by any vertex outside $T_{v}$ and $T_{w}$,
respectively.}%
\label{Fig_branching}%
\end{figure}

Now, we will introduce a property called "branch-resolving" which a set of
vertices $S\subseteq V(G)$ must possess in order to avoid this problem of non
distinguished vertices (resp. edges) due to branching. First, a thread hanging
at a vertex $v$ of degree $\geq3$ is $S$\emph{-free} if it does not contain a
vertex from $S.$ Now, a set of vertices $S\subseteq V(G)$ is
\emph{branch-resolving} if at most one $S$-free thread is hanging at every
vertex $v\in V(G)$ of degree $\geq3$. Therefore, for every branch-resolving
set $S$ it holds that $\left\vert S\right\vert \geq L(G)$ where
\[
L(G)=\sum_{v\in V(G),\ell(v)>1}(\ell(v)-1).
\]
It is known in literature \cite{TratnikEdge, KhullerVertex} that for a tree
$T$ it holds that $\mathrm{dim}(G)=\mathrm{edim}(G)=L(G).$

Also, given a set of vertices $S\subseteq V(G),$ we say that a vertex $v\in
V(C_{i})$ is $S$\emph{-active} if $T_{v}$ contains a vertex from $S.$ The
number of $S$-active vertices on a cycle $C_{i}$ is denoted by $a_{S}(C_{i}).$
If $a_{S}(C_{i})\geq2$ for every cycle $C_{i}$ in $G,$ then we say the set $S$
is \emph{biactive}. For a biactive branch-resolving set $S$ the following
holds: if a vertex $v$ from a cycle $C_{i}$ is branch-active, then $T_{v}$
contains a vertex with two threads hanging at it or $T_{v}$ contains a cycle,
either way $T_{v}$ contains a vertex from $S,$ so $v$ is $S$-active.
Therefore, for a biactive branch-resolving set $S$ we have $a_{S}(C_{i})\geq
b(C_{i})$ for every $i$.

\begin{lemma}
\label{Lemma_biactive_branchResolving}Let $G$ be a cactus graph and let
$S\subseteq V(G)$ be a set of vertices in $G.$ If $S$ is a vertex (resp. an
edge) metric generator, then $S$ is a biactive branch-resolving set.
\end{lemma}

\begin{proof}
Suppose to the contrary that a vertex (resp. an edge) metric generator $S$ is
not a biactive branch-resolving set. If $S$ is not branch-resolving, then
there exists a vertex $v$ of degree $\geq3$ and two threads hanging at $v$
which do not contain a vertex from $S.$ Let $v_{1}$ and $v_{2}$ be two
neighbors of $v,$ each belonging to one of these two threads. Then $v_{1}$ and
$v_{2}$ (resp. $v_{1}v$ and $v_{2}v$) are not distinguished by $S,$ a contradiction.

Assume now that $S$ is not biactive. We may assume that $G$ has at least one
cycle, otherwise $G$ is a tree and there is nothing to prove. Notice that an
empty set $S$ cannot be either a vertex or an edge metric generator in a
cactus graph unless $G=K_{2}$ but then it is a tree. Therefore, if $S$ is not
biactive, there must exist a cycle $C_{i}$ with precisely one $S$-active
vertex $v$ and let $v_{1}$ and $v_{2}$ be the two neighbors of $v$ on $C_{i}.$
Then $v_{1}$ and $v_{2}$ (resp. $v_{1}v$ and $v_{2}v$) are not distinguished
by $S,$ a contradiction.
\end{proof}

The above lemma gives us a necessary condition for $S$ to be a vertex (resp.
an edge) metric generator in a cactus graph. In \cite{SedSkreUnicyclic}, a
more elaborate condition for unicyclic graphs was established, which is both
necessary and sufficient. In this paper we will extend that condition to
cactus graphs, but to do so we first need to introduce the following
definitions from \cite{SedSkreUnicyclic}. Let $C_{i}$ be a cycle in a cactus
graph $G$ and let $v_{i},$ $v_{j}$ and $v_{k}$ be three vertices of $C_{i},$
we say that $v_{i},$ $v_{j}$ and $v_{k}$ are a \emph{geodesic triple} on
$C_{i}$ if
\[
d(v_{i},v_{j})+d(v_{j},v_{k})+d(v_{i},v_{k})=|V(C_{i})|.
\]
It was shown in \cite{SedSkreBounds} that a biactive branch-resolving set with
a geodesic triple of $S$-active vertices on every cycle is both a vertex and
an edge metric generator. This result is useful for bounding the dimensions
from above. Also, we need the definition of the five graph configurations from
\cite{SedSkreUnicyclic}.

\begin{figure}[ph]
\begin{center}
$%
\begin{array}
[c]{ll}%
\text{a) \raisebox{-1\height}{\includegraphics[scale=0.8]{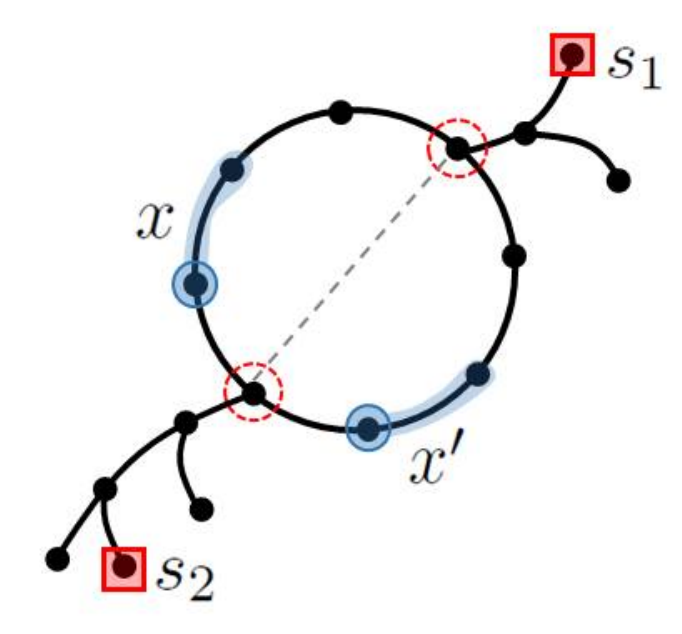}}} &
\text{b) \raisebox{-1\height}{\includegraphics[scale=0.8]{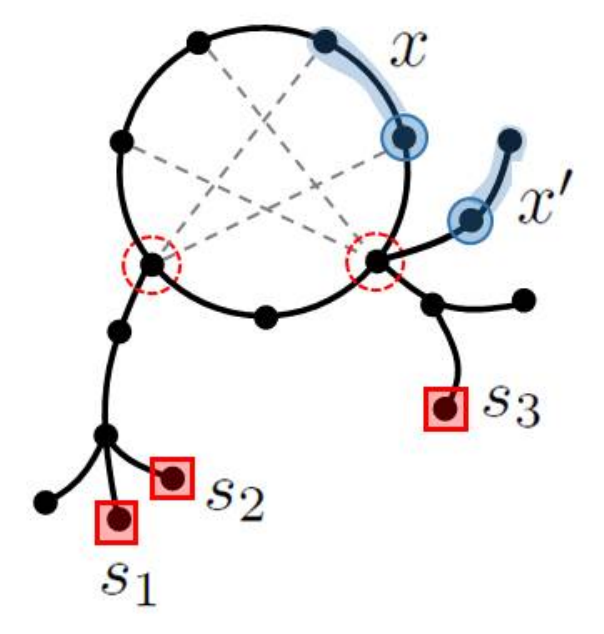}}}\\
\text{c) \raisebox{-1\height}{\includegraphics[scale=0.8]{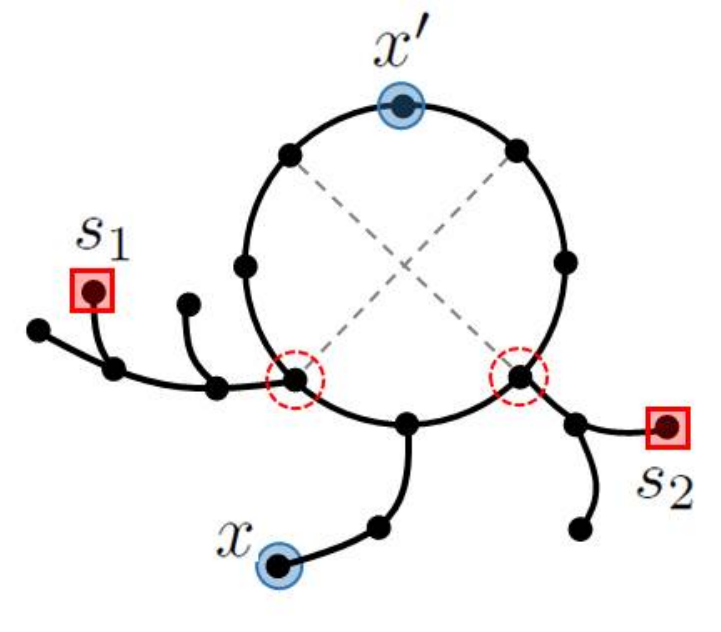}}} &
\text{d) \raisebox{-1\height}{\includegraphics[scale=0.8]{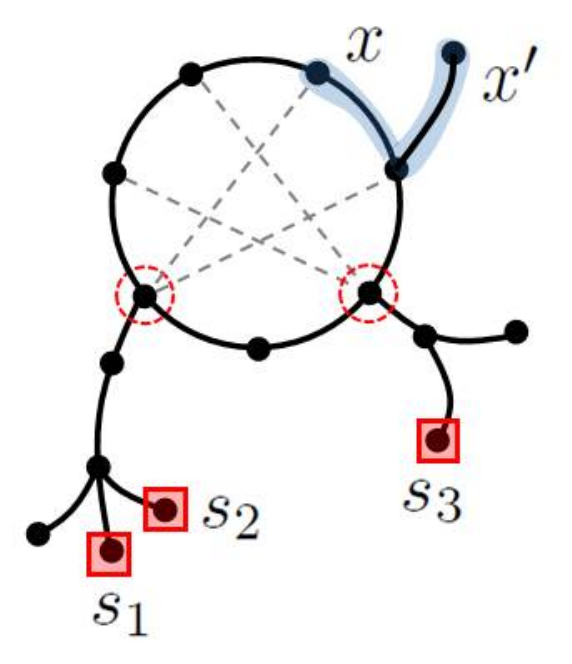}}}\\
\text{e) \raisebox{-1\height}{\includegraphics[scale=0.8]{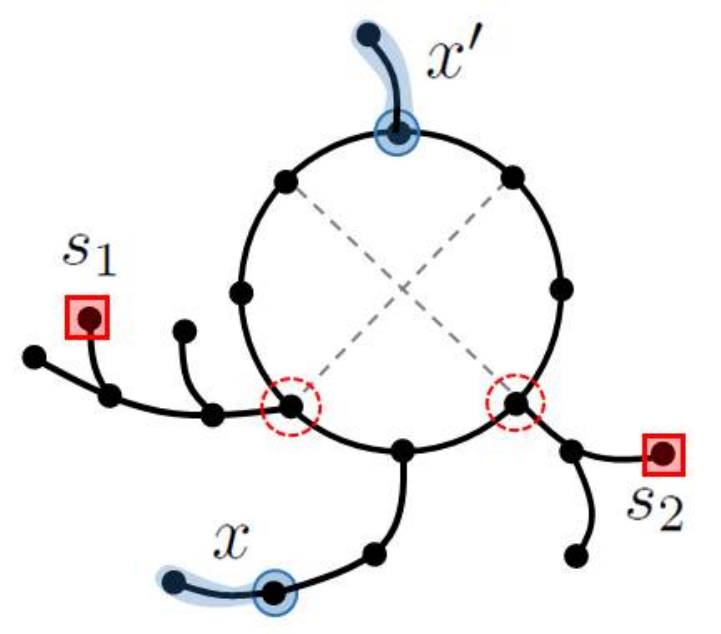}}} &
\text{f) \raisebox{-1\height}{\includegraphics[scale=0.8]{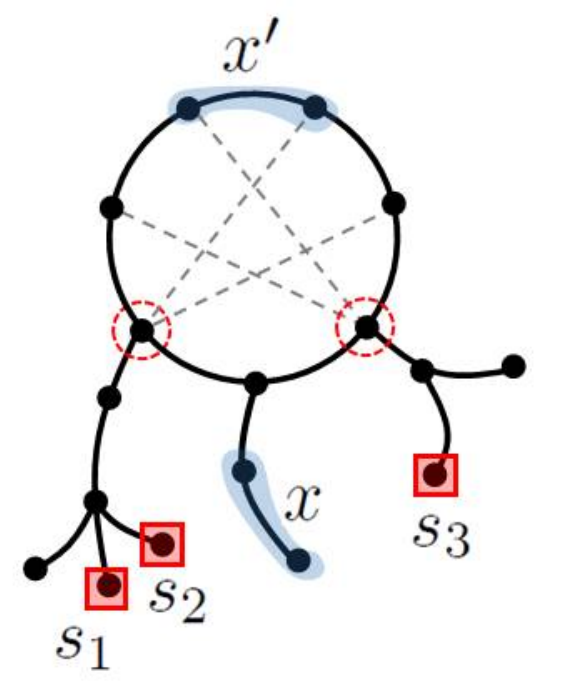}}}%
\end{array}
$
\end{center}
\caption{All six graphs shown in this figure are unicyclic graphs with a
biactive branch-resolving set $S\ $comprised of vertices $s_{i}$. Vertices on
a cycle which are $S$-active are marked by a dashed circle and connected to
its antipodal vertices by a dashed line. Each graph in this figure contains at
least one of the five configurations, namely: a) $\mathcal{A}$, b)
$\mathcal{B}$ and also $\mathcal{D},$ c) $\mathcal{C}$, d) $\mathcal{D},$ e)
$\mathcal{E}$ on even cycle and also $\mathcal{C}$, f) $\mathcal{E}$ on odd
cycle. A pair $x$ and $x^{\prime}$ of undistinguished vertices and/or edges is
highlighted in each of the graphs. Notice that the graph in b) which contains
$\mathcal{B}$ consequently also contains $\mathcal{D},$ but graph in d) which
contains $\mathcal{D}$ does not contain $\mathcal{B}.$ Similarly, the graph in
e) which contains $\mathcal{E}$ also contains $\mathcal{C},$ but the graph
from c) which contains $\mathcal{C}$ does not contain $\mathcal{E}$.}%
\label{Fig_configurations}%
\end{figure}

\begin{definition}
Let $G$ be a cactus graph, $C$ a cycle in $G$ of the length $g$, and $S$ a
biactive branch-resolving set in $G$. We say that $C=v_{0}v_{1}\cdots v_{g-1}$
is \emph{canonically labeled} with respect to $S$ if $v_{0}$ is $S$-active and
$k=\max\{i:v_{i}$ is $S$-active$\}$ is as small as possible.
\end{definition}

Let us now introduce five configurations which a cactus graph can contain with
respect to a biactive branch-resolving set $S.$ All these configurations are
illustrated by Figure \ref{Fig_configurations}.\ 

\begin{definition}
Let $G$ be a cactus graph, $C$ a canonically labeled cycle in $G$ of the
length $g$, and $S$ a biactive branch-resolving set in $G$. We say that the
cycle $C$ \emph{with respect} to $S$ \emph{contains} configurations:

\begin{description}
\item {$\mathcal{A}$}. If $a_{S}(C)=2$, $g$ is even, and $k=g/2$;

\item {$\mathcal{B}$}. If $k\leq\left\lfloor g/2\right\rfloor -1$ and there is
an $S$-free thread hanging at a vertex $v_{i}$ for some $i\in\lbrack
k,\left\lfloor g/2\right\rfloor -1]\cup\lbrack\left\lceil g/2\right\rceil
+k+1,g-1]\cup\{0\}$;

\item {$\mathcal{C}$}. If $a_{S}(C)=2$, $g$ is even, $k\leq g/2$ and there is
an $S$-free thread of the length $\geq g/2-k$ hanging at a vertex $v_{i}$ for
some $i\in\lbrack0,k]$;

\item $\mathcal{D}$. If $k\leq\left\lceil g/2\right\rceil -1$ and there is an
$S$-free thread hanging at a vertex $v_{i}$ for some $i\in\lbrack
k,\left\lceil g/2\right\rceil -1]\cup\lbrack\left\lfloor g/2\right\rfloor
+k+1,g-1]\cup\{0\}$;

\item {$\mathcal{E}$}. If $a_{S}(C)=2$ and there is an $S$-free thread of the
length $\geq\left\lfloor g/2\right\rfloor -k+1$ hanging at a vertex $v_{i}$
with $i\in\lbrack0,k].$ Moreover, if $g$ is even, an $S$-free thread must be
hanging at the vertex $v_{j}$ with $j=g/2+k-i$.
\end{description}
\end{definition}

\begin{figure}[h]
\begin{center}
\includegraphics[scale=0.9]{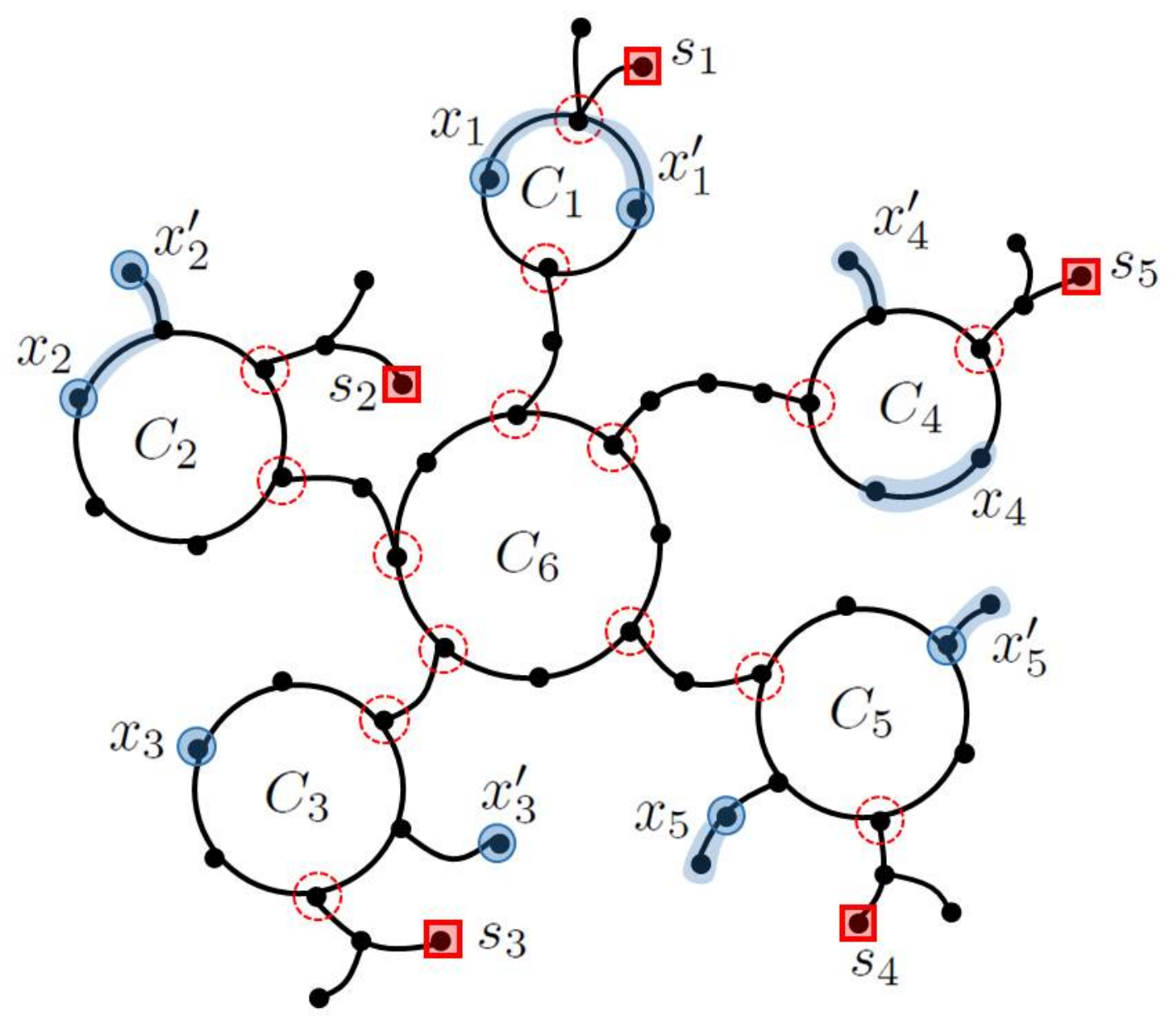}
\end{center}
\caption{A cactus graph $G$ from Example \ref{Example_conf}. A smallest
biactive branch-resolving set $S=\{s_{1},s_{2},s_{3},s_{4},s_{5}\}$ is marked
in the figure by squares. Vertices on the cycles which are $S$-active are
marked by a dashed circle. Since every cycle contains a configuration with
respect to $S,$ for each cycle $C_{i}$ there is a pair of vertices and/or
edges $x_{i}$ and $x_{i}^{\prime}$ which are not distinguished by $S,$ these
pairs are also highlighted in the figure.}%
\label{Fig_configExample}%
\end{figure}

Notice that only an even cycle can contain configuration $\mathcal{A}$ or
$\mathcal{C}$. Also, configurations $\mathcal{B}$ and $\mathcal{D}$ are almost
the same, they differ only if $C$ is odd where the index $i$ can take two more
values in $\mathcal{D}$ than in $\mathcal{B}.$ Finally, for configurations
$\mathcal{A}$, $\mathcal{C}$, and $\mathcal{E}$ it holds that $a_{S}(C)=2,$ so
there are only two $S$-active vertices on the cycle $C$ and hence no geodesic
triple of $S$-active vertices. On the other hand, for configurations
$\mathcal{B}$ and $\mathcal{D}$ there might be more than two $S$-active
vertices on the cycle $C,$ but the bounds $k\leq\left\lfloor g/2\right\rfloor
-1$ and $k\leq\left\lceil g/2\right\rceil -1$ again imply there is no geodesic
triple of $S$-active vertices on $C.$ Therefore, we can state the following
observation which is useful for constructing metric generators.

\begin{remark}
\label{Obs_geodTriple}If there is a geodesic triple of $S$-active vertices on
a cycle $C$ of a cactus graph $G,$ then $C$ does not contain any of the
configurations $\mathcal{A}$, $\mathcal{B}$, $\mathcal{C}$, $\mathcal{D}$, and
$\mathcal{E}$ with respect to $S.$
\end{remark}

The following result regarding configurations $\mathcal{A}$, $\mathcal{B}$,
$\mathcal{C}$, $\mathcal{D}$, and $\mathcal{E}$ was established for unicyclic
graphs (see Lemmas 6, 7, 13 and 14 from \cite{SedSkreUnicyclic}).

\begin{theorem}
\label{Lemma_configurations}Let $G$ be a unicyclic graph with the cycle $C$
and let $S$ be a biactive branch-resolving set in $G$. The set $S$ is a vertex
(resp. an edge) metric generator if and only if $C$ does not contain any of
configurations $\mathcal{A}$, $\mathcal{B}$, and $\mathcal{C}$ (resp.
$\mathcal{A}$, $\mathcal{D}$, and $\mathcal{E}$) with respect to $S.$
\end{theorem}

In this paper we will extend this result to cactus graphs and then use it to
determine the exact value of the vertex and the edge metric dimensions of such
graphs. We first give an example how this approach with configurations can be
extended to cactus graphs.

\begin{example}
\label{Example_conf}Let $G$ be the cactus graph from Figure
\ref{Fig_configExample}. The graph $G$ contains six cycles and the set
$S=\{s_{1},s_{2},s_{3},s_{4},s_{5}\}$ is a smallest biactive branch-resolving
set in $G$. In the figure the set of $S$-active vertices on each cycle is
marked by a dashed circle. The cycle $C_{1}$ (resp. $C_{2}$, $C_{3}$, $C_{4}$,
$C_{5}$) with respect to $S$ contains configuration $\mathcal{A}$ (resp.
$\mathcal{B}$ and also $\mathcal{D}$, $\mathcal{C}$, $\mathcal{E}$ on odd
cycle, $\mathcal{E}$ on even cycle), so in each of these cycles there is a
pair of vertices and/or edges $x_{i}$ and $x_{i}^{\prime}$ which is not
distinguished by $S.$ The cycle $C_{6}$ does not contain any of the five
configurations as there is a geodesic triple of $S$-active vertices on
$C_{6},$ so all pairs of vertices and all pair of edges in $C_{6}$ are
distinguished by $S$.
\end{example}

\begin{figure}[h]
\begin{center}
\includegraphics[scale=0.9]{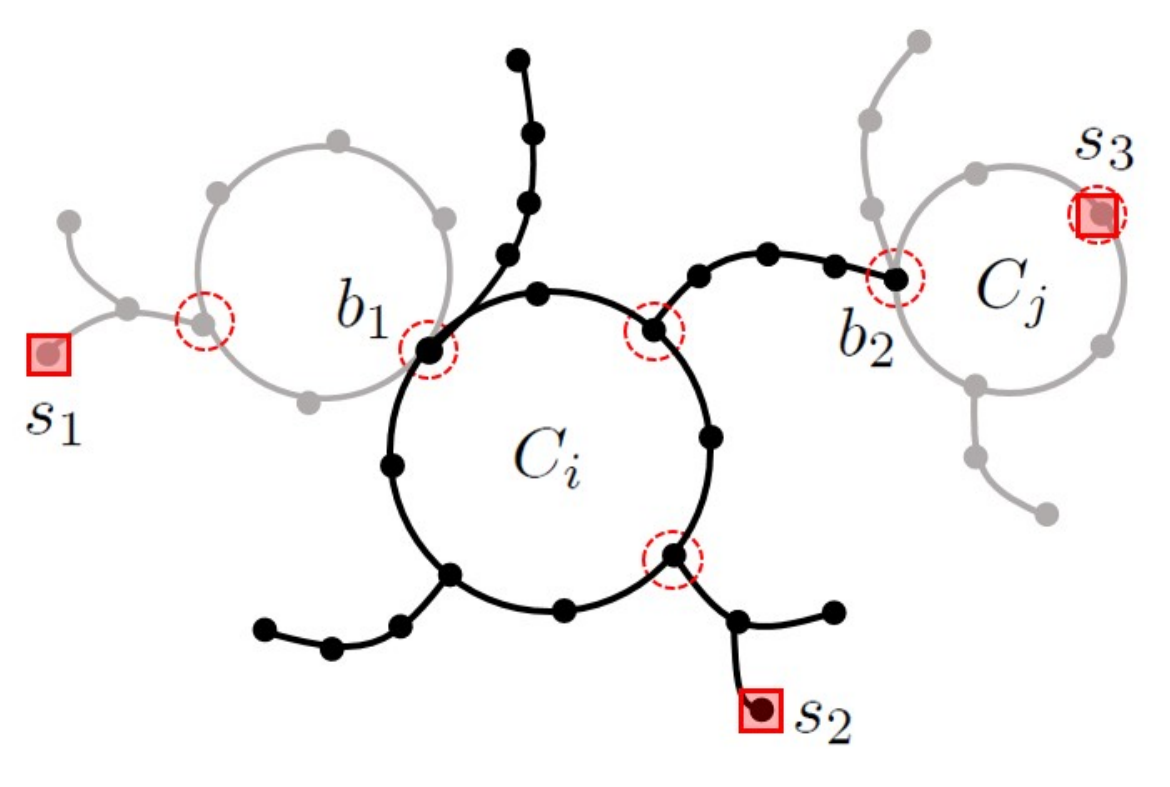}
\end{center}
\caption{A cactus graph $G$ with three cycles in which the unicyclic region
$G_{i}$ of the cycle $C_{i}$ is distinguished with $b_{1}$ and $b_{2}$ being
the boundary vertices of $G_{i}$. The set $S=\{s_{1},s_{2},s_{3}\}$ is a
smallest biactive branch-resolving set in $G$ for which a set of $S$-active
vertices is marked on each cycle. For the set $S,$ the regional set in $G_{i}$
is $S_{i}=\{s_{2},b_{1},b_{2}\}.$}%
\label{Fig_region}%
\end{figure}

Besides this configuration approach, in cactus graphs an additional condition
will have to be introduced for the situation when a pair of cycles share a vertex.

\section{Metric generators in cacti}

Let $G$ be a cactus graph with cycles $C_{1},\ldots,C_{c}.$ We say that a
vertex $v\in V(G)$ \emph{gravitates} to a cycle $C_{i}$ in $G$ if there is a
path from $v$ to a vertex from $C_{i}$ which does not share any edge nor any
internal vertex with any cycle of $G.$ A \emph{unicyclic region} of the cycle
$C_{i}$ from $G$ is the subgraph $G_{i}$ of $G$ induced by all vertices that
gravitate to $C_{i}$ in $G.$ The notion of unicyclic region of a cactus graph
is illustrated by Figure \ref{Fig_region}.

Notice that each unicyclic region $G_{i}$ is a unicyclic graph with its cycle
being $C_{i}.$ Also, considering the example from Figure \ref{Fig_region}, one
can easily notice that two distinct unicyclic regions may not be vertex
disjoint, as the path connecting vertex $b_{2}$ and the cycle $C_{i}$ belongs
both to $G_{i}$ and $G_{j}.$ But, it does hold that all unicyclic regions
cover the whole $G.$ We say that a subgraph $H$ of a graph $G$ is an
\emph{isometric} subgraph if $d_{H}(u,v)=d_{G}(u,v)$ for every pair of
vertices $u,v\in V(H).$ The following observation is obvious.

\begin{remark}
The unicyclic region $G_{i}$ of a cycle $C_{i}$ is an isometric subgraph of
$G.$
\end{remark}

Finally, we say that a vertex $v$ from a unicyclic region $G_{i}$ is a
\emph{boundary} vertex if $v\in V(C_{j})$ for $j\not =i.$ In the example from
Figure \ref{Fig_region}, the boundary vertices of the region $G_{i}$ are
$b_{1}$ and $b_{2}$.

Let $S$ be a biactive branch-resolving set in $G$ and let $G_{i}$ be a
unicyclic region in $G.$ For the set $S$ we define the \emph{regional set}
$S_{i}$ as the set obtained from $S\cap V(G_{i})$ by introducing all boundary
vertices from $G_{i}$ to $S$. For example, in Figure \ref{Fig_region} the set
$S_{i}=\{s_{2},b_{1},b_{2}\}$ is the regional set in the region of the cycle
$C_{i}$.

\begin{lemma}
\label{Lemma_regionGenerator}Let $G$ be a cactus graph with $c$ cycles
$C_{1},\ldots,C_{c}$ and let $S\subseteq V(G)$. If $S$ is a vertex (resp.
edge) metric generator in $G,$ then the regional set $S_{i}$ is a vertex
(resp. edge) metric generator in the unicyclic region $G_{i}$ for every
$i\in\{1,\ldots,c\}$.
\end{lemma}

\begin{proof}
Suppose first that there is a cycle $C_{i}$ in $G$ such that the regional set
$S_{i}$ is not a vertex (resp. an edge) metric generator in the unicyclic
region $G_{i}.$ This implies that there exists a pair of vertices (resp.
edges) $x$ and $x^{\prime}$ in $G_{i}$ which are not distinguished by $S_{i}.$
We will show that $x$ and $x^{\prime}$ are not distinguished by $S$ in $G$
either. Suppose the contrary, i.e. there is a vertex $s\in S$ which
distinguishes $x$ and $x^{\prime}$ in $G.$ If $s\in V(G_{i}),$ then $s\in
S_{i}.$ Since $G_{i}$ is an isometric subgraph of $G,$ then $x$ and
$x^{\prime}$ would be distinguished by $s\in S_{i}$ in $G_{i},$ a
contradiction. Assume therefore that $s\not \in V(G_{i}).$ Notice that the
shortest path from every vertex (resp. edge) in $G_{i}$ to $s$ leads through a
same boundary vertex $b$ in $G_{i}.$ The definition of $S_{i}$ implies $b\in
S_{i},$ so $x$ and $x^{\prime}$ are not distinguished by $b.$ Therefore, we
obtain%
\[
d(x,s)=d(x,b)+d(b,s)=d(x^{\prime},b)+d(b,s)=d(x^{\prime},s),
\]
so $x$ and $x^{\prime}$ are not distinguished by $s$ in $G,$ a contradiction.
\end{proof}

Notice that the condition from Lemma \ref{Lemma_regionGenerator} is necessary
for $S$ to be a metric generator, but it is not sufficient as is illustrated
by the graph shown in Figure \ref{Fig_incidence} in which every regional set
$S_{i}$ is a vertex (resp. an edge) metric generator in the corresponding
region $G_{i}$, but there still exists a pair of vertices (resp. edges) which
is not distinguished by $S,$ so $S$ is not a vertex (resp. an edge) metric
generator in $G$.

\begin{figure}[h]
\begin{center}
\includegraphics[scale=0.9]{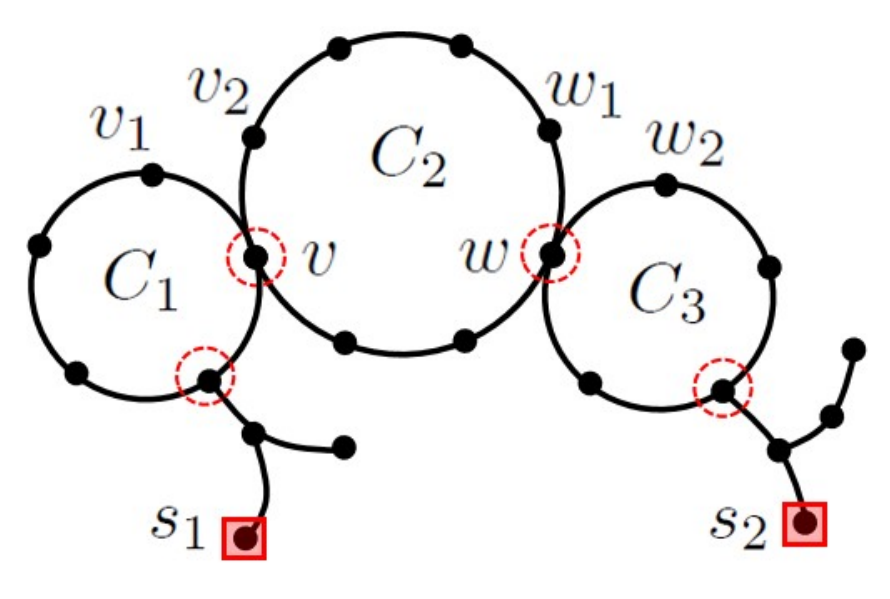}
\end{center}
\caption{A cactus graph $G$ with three cycles and a smallest biactive
branch-resolving set $S=\{s_{1},s_{2}\}$ in $G.$ Every regional set $S_{i}$ is
a vertex (resp. an edge) metric generator in the corresponding region $G_{i}$,
but still a pair of vertices $v_{1}$ and $v_{2}$ is not distinguished by $S$.
The pair of edges $v_{1}v$ and $v_{2}v$ is also not distinguished by $S$ and
the same holds for the pair of edges $w_{1}w$ and $w_{2}w.$}%
\label{Fig_incidence}%
\end{figure}

Next, we will introduce notions which are necessary to state a condition which
will be both necessary and sufficient for a biactive branch-resolving set $S$
to be a vertex (resp. an edge) metric generator in a cactus graph $G.$ An
$S$\emph{-path} of the cycle $C_{i}$ is any subpath of $C_{i}$ which contains
all $S$-active vertices on $C_{i}$ and is of minimum possible length. We
denote an $S$\emph{-}path of the cycle $C_{i}$ by $P_{i}$. Notice that the
end-vertices of an $S$-path are always $S$-active, otherwise it would not be
shortest. For example, on the cycle $C_{2}$ in Figure \ref{Fig_incidence}
there are two different paths connecting $S$-active vertices $v$ and $w,$ one
is of the length $3$ and the other of length $5,$ so the shorter one is an
$S$-path. Also, an $S$\emph{-}path $P_{i}$ of a cycle $C_{i}$ may not be
unique, as there may exist several shortest subpaths of $C_{i}$ containing all
$S$-active vertices on $C_{i},$ but if the length of $P_{i}$ satisfies
$\left\vert P_{i}\right\vert \leq\left\lceil g_{i}/2\right\rceil -1$ then
$P_{i}$ is certainly unique and its end-vertices are $v_{0}$ and $v_{k}$ in
the canonical labelling of $C_{i}.$

\begin{definition}
Let $G$ be a cactus graph with cycles $C_{1},\ldots,C_{c}$ and let $S$ be a
biactive branch-resolving set in $G.$ We say that a vertex $v\in V(C_{i})$ is
\emph{vertex-critical} (resp. \emph{edge-critical}) on $C_{i}$ with respect to
$S$ if $v$ is an end-vertex of $P_{i}$ and $\left\vert P_{i}\right\vert
\leq\left\lfloor g_{i}/2\right\rfloor -1$ (resp. $\left\vert P_{i}\right\vert
\leq\left\lceil g_{i}/2\right\rceil -1$).
\end{definition}

Notice that the notion of a vertex-critical and an edge-critical vertex
differs only on odd cycles. We say that two distinct cycles $C_{i}$ and
$C_{j}$ of a cactus graph $G$ are \emph{vertex-critically incident} (resp.
\emph{edge-critically incident}) with respect to $S$ if $C_{i}$ and $C_{j}$
share a vertex $v$ which is vertex-critical (resp. edge-critical) with respect
to $S$ on both $C_{i}$ and $C_{j}$. Notice that on odd cycles the required
length of an $S$-path $P_{i}$ for $v$ to be vertex-critical differs from the
one required for $v$ to be edge-critical, while on even cycles the required
length is the same.

To illustrate this notion, let us consider the cycle $C_{2}$ in the graph from
Figure \ref{Fig_incidence}. Vertices $v$ and $w$ are both vertex-critical and
edge-critical on $C_{2}$ with respect to $S$ from the figure. Vertex $v$
belongs also to $C_{1}$ and it is also both vertex- and edge-critical on
$C_{1}.$ Therefore, cycles $C_{1}$ and $C_{2}$ are both vertex- and
edge-critically incident, the consequence of which is that a pair of vertices
$v_{1}$ and $v_{2}$ which are neighbors of $v$ and a pair of edges $v_{1}v$
and $v_{2}v$ which are incident to $v$ are not distinguished by $S.$ On the
other hand, vertex $w$ belongs also to $C_{3}$ on which it is edge-critical,
but it is not vertex-critical since $P_{3}$ is not short enough. So, $C_{2}$
and $C_{3}$ are edge-critically incident, but not vertex-critically incident.
Consequently, a pair of edges $w_{1}w$ and $w_{2}w$ is not distinguished by
$S,$ but a pair of vertices $w_{1}$ and $w_{2}$ is distinguished by $S.$ We
will show in the following lemma that this holds in general.

\begin{lemma}
\label{Lemma_incidence}Let $G$ be a cactus graph with $c$ cycles $C_{1}%
,\ldots,C_{c}$ and let $S$ be a biactive branch-resolving set in $G$. If $S$
is a vertex (resp. an edge) metric generator in $G,$ then there is no pair of
cycles in $G$ which are vertex-critically (resp. edge critically) incident
with respect to $S$.
\end{lemma}

\begin{proof}
Let $S$ be a vertex (resp. an edge) metric generator in $G.$ Suppose the
contrary, i.e. there are two distinct cycles $C_{i}$ and $C_{j}$ in $G$ which
are vertex-critically (resp. edge-critically) incident with respect to $S$.
This implies that $C_{i}$ and $C_{j}$ share a vertex $v$ which is
vertex-critical (resp. edge-critical) on both $C_{i}$ and $C_{j}$. Let $x$ and
$x^{\prime}$ be a pair of vertices (resp. edges) which are neighbors (resp.
incident) to $v$ on cycles $C_{i}$ and $C_{j}$ respectively, but which are not
contained on paths $P_{i}$ and $P_{j}.$ The length of paths $P_{i}$ and
$P_{j}$ which is required by the definition of a vertex-critical (resp.
edge-critical) vertex implies that a shortest path from both $x$ and
$x^{\prime}$ to all vertices from $P_{i}$ and $P_{j}$ leads through $v.$ Since
$P_{i}$ and $P_{j}$ contain all $S$-active vertices on $C_{i}$ and $C_{j},$
this further implies that a shortest path from both $x$ and $x^{\prime}$ to
all vertices from $S$ leads through $v.$ Since $d(x,v)=d(x^{\prime},v)$, it
follows that $x$ and $x^{\prime}$ are not distinguished by $S,$ a contradiction.
\end{proof}

Each of Lemmas \ref{Lemma_regionGenerator} and \ref{Lemma_incidence} gives a
necessary condition for a biactive branch-resolving set $S$ to be a vertex
(resp. an edge) metric generator in a cactus graph $G$. Let us now show that
these two necessary conditions taken together form a sufficient condition for
$S$ to be a vertex (resp. an edge) metric generator.

\begin{lemma}
\label{Lemma_sufficient}Let $G$ be a cactus graph with $c$ cycles
$C_{1},\ldots,C_{c}$ and let $S$ be a biactive branch-resolving set in $G$. If
a regional set $S_{i}$ is a vertex (resp. an edge) metric generator in the
unicyclic region $G_{i}$ for every $i=1,\ldots,c$ and there are no
vertex-critically (resp. edge-critically) incident cycles in $G,$ then $S$ is
a vertex (resp. an edge) metric generator in $G.$
\end{lemma}

\begin{proof}
Let $x$ and $x^{\prime}$ be a pair of vertices (resp. edges) from $G.$ We want
to show that $S$ distinguishes $x$ and $x^{\prime}.$ In order to do so, we
distinguish the following two cases.

\medskip\noindent\textbf{Case 1:}\emph{ }$x$\emph{ and }$x^{\prime}$\emph{
belong to a same unicyclic region }$G_{i}$\emph{ of }$G.$ Since the regional
set $S_{i}$ is a vertex (resp. an edge) metric generator in $G_{i},$ there is
a vertex $s\in S_{i}$ which distinguishes $x$ and $x^{\prime}$ in $G_{i}.$ If
$s\in S,$ then the fact that $G_{i}$ is an isometric subgraph of $G$ implies
that the pair $x$ and $x^{\prime}$ is distinguished by the same $s$ in $G$ as
well. Assume therefore that $s\not \in S,$ so the definition of the regional
set $S_{i}$ implies that $s$ is a boundary vertex of $G_{i}.$ Let $s^{\prime
}\in S$ be a vertex in $G$ such that the shortest path from $s^{\prime}$ to
both $x$ and $x^{\prime}$ leads through the boundary vertex $s.$ Recall that
such a vertex $s^{\prime}$ must exist since $S$ is biactive. The fact that $s$
distinguishes $x$ and $x^{\prime}$ in $G_{i}$, implies $d(x,s)\not =%
d(x^{\prime},s),$ which further implies%
\[
d(x,s^{\prime})=d(x,s)+d(s,s^{\prime})\not =d(x^{\prime},s)+d(s,s^{\prime
})=d(x^{\prime},s^{\prime}),
\]
so the pair $x$ and $x^{\prime}$ is distinguished by $S$ in $G$.

\medskip\noindent\textbf{Case 2:} $x$\emph{ and }$x^{\prime}$\emph{ do not
belong to a same unicyclic region of }$G.$ Let us assume that $x$ belongs to
$G_{i}$ and $x^{\prime}$ does not belong to $G_{i},$ and say it belongs to
$G_{j}$ for $j\not =i$. If $x$ and $x^{\prime}$ are distinguished by a vertex
$s\in S\cap V(G_{i}),$ then the claim is proven, so let us assume that $x$ and
$x^{\prime}$ are not distinguished by any $s\in S\cup V(G_{i}).$ Since $x$ and
$x^{\prime}$ do not belong to a same unicyclic region, there exists a boundary
vertex $b$ of the unicyclic region $G_{i}$ such that the shortest path from
$x$ to $x^{\prime}$ leads through $b.$ Let $s_{b}$ be a vertex from $S$ such
that the shortest path from $x$ to $s_{b}$ also leads through $b,$ which must
exist since $S$ is biactive. We want to prove that $x$ and $x^{\prime}$ are
distinguished by $s_{b}.$ Let us suppose the contrary, i.e. $d(x,s_{b}%
)=d(x^{\prime},s_{b}).$ Then we have the following%
\[
d(x,b)+d(b,s_{b})=d(x,s_{b})=d(x^{\prime},s_{b})\leq d(x^{\prime}%
,b)+d(b,s_{b}),
\]
from which we obtain
\begin{equation}
d(x,b)\leq d(x^{\prime},b). \label{For_riste}%
\end{equation}
Now, we distinguish the following two subcases.

\medskip\noindent\textbf{Subcase 2.a:} $b$\emph{ does not belong to }%
$V(C_{i}).$ Notice that by the definition of the unicyclic region, any acyclic
structure hanging at $b$ in $G$ is not included in $G_{i}$, as is illustrated
by $b_{2}$ from Figure \ref{Fig_region}, which implies $b$ is a leaf in
$G_{i}.$ Let $b^{\prime}$ be the only neighbor of $b$ in $G_{i}.$ The
inequality (\ref{For_riste}) further implies $d(x,b^{\prime})<d(x^{\prime
},b^{\prime})$ since $x$ belongs to $G_{i}$ and $x^{\prime}$ does not. Let
$v_{0}$ be the vertex from $C_{i}$ closest to $b,$ which implies $v_{0}$ is
$S$-active on $C_{i}.$ Let $v_{k}$ be an $S$-active vertex on $C_{i}$ distinct
from $v_{0},$ such a vertex $v_{k}$ must exist on $C_{i}$ because we assumed
$S$ is biactive. So, we have
\[
d(x,v_{k})\leq d(x,b^{\prime})+d(b^{\prime},v_{k})<d(x^{\prime},b^{\prime
})+d(b^{\prime},v_{k})=d(x^{\prime},v_{k}).
\]
Let $s_{k}$ be a vertex from $S$ which belongs to the connected component of
$G-E(C_{i})$ containing $v_{k}.$ Then we have%
\[
d(x,s_{k})\leq d(x,v_{k})+d(v_{k},s_{k})<d(x^{\prime},v_{k})+d(v_{k}%
,s_{k})=d(x^{\prime},s_{k}).
\]
Therefore, $S$ distinguishes $x$ and $x^{\prime},$ so $S$ is a vertex (resp.
an edge) metric generator in $G.$

\medskip\noindent\textbf{Subcase 2.b:} $b$\emph{ belongs to }$V(C_{i}).$ Since
$b$ is a boundary vertex of the unicyclic region $G_{i},$ this implies there
is a cycle $C_{l},$ for $l\not =i,$ such that $b\in V(C_{l})$. Therefore,
cycles $C_{i}$ and $C_{l}$ share the vertex $b.$ Notice that any acyclic
structure hanging at $b$ belongs to both $G_{i}$ and $G_{l}$, as is
illustrated by $b_{1}$ from Figure \ref{Fig_region}. If $x^{\prime}$ belongs
to $G_{l},$ then neither $x$ nor $x^{\prime}$ can belong to an acyclic
structure hanging at $b,$ as we assumed that $x$ and $x^{\prime}$ do not
belong to a same component. On the other hand, if $x^{\prime}$ does not belong
to $G_{l}$ and $x$ belongs to an acyclic structure hanging at $b,$ then we
switch $G_{i}$ by $G_{l}$ and assume that $x$ belongs to $G_{l}.$ This way we
assure that neither $x$ not $x^{\prime}$ belong to an acyclic structure
hanging at $b.$

If $d(x,b)<d(x^{\prime},b),$ let $v_{k}$ be an $S$-active vertex on $C_{i}$
distinct from $b,$ which must exist as $S$ is biactive. From (\ref{For_riste})
we obtain%
\[
d(x,v_{k})\leq d(x,b)+d(b,v_{k})<d(x^{\prime},b)+d(b,v_{k})=d(x^{\prime}%
,v_{k}),
\]
so similarly as in previous subcase $x$ and $x^{\prime}$ are distinguished by
a vertex $s_{k}\in S$ which belongs to the connected component of $G-E(C_{i})$
which contains $v_{k}.$

Therefore, assume that $d(x,b)=d(x^{\prime},b).$ If a shortest path from $x$
to all $S$-active vertices on $C_{i}$ and a shortest path from $x^{\prime}$ to
all $S$-active vertices on $C_{l}$ leads through $b,$ then $x^{\prime}$
belongs to $C_{l},$ i.e. $j=l$, and the pair of cycles $C_{i}$ and $C_{l}$ are
vertex-critically (resp. edge-critically) incident, a contradiction. So, we
may assume there is an $S$-active vertex $v_{k}$, say on $C_{i}$, such that a
shortest path from $x$ to $v_{k}$ does not lead through $b.$ Therefore,
\[
d(x,v_{k})<d(x,b)+d(b,v_{k})=d(x^{\prime},b)+d(b,v_{k})=d(x^{\prime},v_{k}).
\]
But now, similarly as in previous cases we have that $x$ and $x^{\prime}$ are
distinguished by a vertex $s_{k}\in S$ which is contained in the connected
component of $G-E(C_{i})$ containing $v_{k}.$
\end{proof}

Let us now relate these results with configurations $\mathcal{A},$
$\mathcal{B}$, $\mathcal{C}$, $\mathcal{D}$ and $\mathcal{E}$.

\begin{lemma}
\label{Obs_RegionConfiguration}Let $G$ be a cactus graph and let $S$ be a
biactive branch-resolving set in $G.$ A cycle $C_{i}$ of the graph $G$
contains configuration $\mathcal{A}$ (or $\mathcal{B}$ or $\mathcal{C}$ or
$\mathcal{D}$ or $\mathcal{E}$) with respect to $S$ in $G$ if and only if
$C_{i}$ contains the respective configuration with respect to $S_{i}$ in
$G_{i}.$
\end{lemma}

\begin{proof}
Let $G$ be a cactus graph with cycles $C_{1},\ldots,C_{c}$ and let $S$ be a
biactive branch-resolving set in $G.$ Since $S$ is a biactive set, for every
boundary vertex $b$ in a unicyclic region $G_{i},$ there is a vertex $s\in S$
such that the shortest path from $s$ to $C_{i}$ leads through $b,$ as it is
shown in Figure \ref{Fig_region}. This implies that the set of $S$-active
vertices on $C_{i}$ in $G$ is the same as the set of $S_{i}$-active vertices
on $C_{i}$ in $G_{i}.$ Since the presence of configurations $\mathcal{A}$,
$\mathcal{B}$, $\mathcal{C}$, $\mathcal{D}$ and $\mathcal{E}$ on a cycle
$C_{i}$, with respect to a set $S,$ by definition depends on the position of
$S$-active vertices on $C_{i},$ the claim follows.
\end{proof}

Notice that Lemmas \ref{Lemma_regionGenerator}, \ref{Lemma_incidence} and
\ref{Lemma_sufficient} give us a condition for $S$ to be a vertex (resp. an
edge) metric generator in a cactus graph, which is both necessary and
sufficient. In the light of Lemma \ref{Obs_RegionConfiguration}, we can
further apply Theorem \ref{Lemma_configurations} to obtain the following
result which unifies all our results.

\begin{theorem}
\label{Cor_generatorCharacterization}Let $G$ be a cactus graph with $c$ cycles
$C_{1},\ldots,C_{c}$ and let $S$ be a biactive branch-resolving set in $G$.
The set $S$ is a vertex (resp. an edge) metric generator if and only if each
cycle $C_{i}$ does not contain any of the configurations $\mathcal{A}$,
$\mathcal{B}$, and $\mathcal{C}$ (resp. $\mathcal{A}$, $\mathcal{D}$, and
$\mathcal{E}$) and there are no vertex-critically (resp. edge-critically)
incident cycles in $G$ with respect to $S$.
\end{theorem}

\begin{proof}
Let $S$ be a vertex (resp. an edge) metric generator in $G$. Then Lemma
\ref{Lemma_regionGenerator} implies that $S_{i}$ is a vertex (resp. edge)
metric generator in the unicyclic region $G_{i}$ and the Theorem
\ref{Lemma_configurations} further implies that every cycle does not contain
any of the configurations $\mathcal{A}$, $\mathcal{B}$, and $\mathcal{C}$
(resp. $\mathcal{A}$, $\mathcal{D}$, and $\mathcal{E}$) with respect to $S.$
Also, Lemma \ref{Lemma_incidence} implies that there are no vertex-critically
(resp. edge-critically) incident cycles in $G$ with respect to $S$.

The other direction is the consequence of Lemma \ref{Lemma_sufficient} and
Theorem \ref{Lemma_configurations}.
\end{proof}

\section{Metric dimensions in cacti}

\begin{figure}[h]
\begin{center}
\includegraphics[scale=0.8]{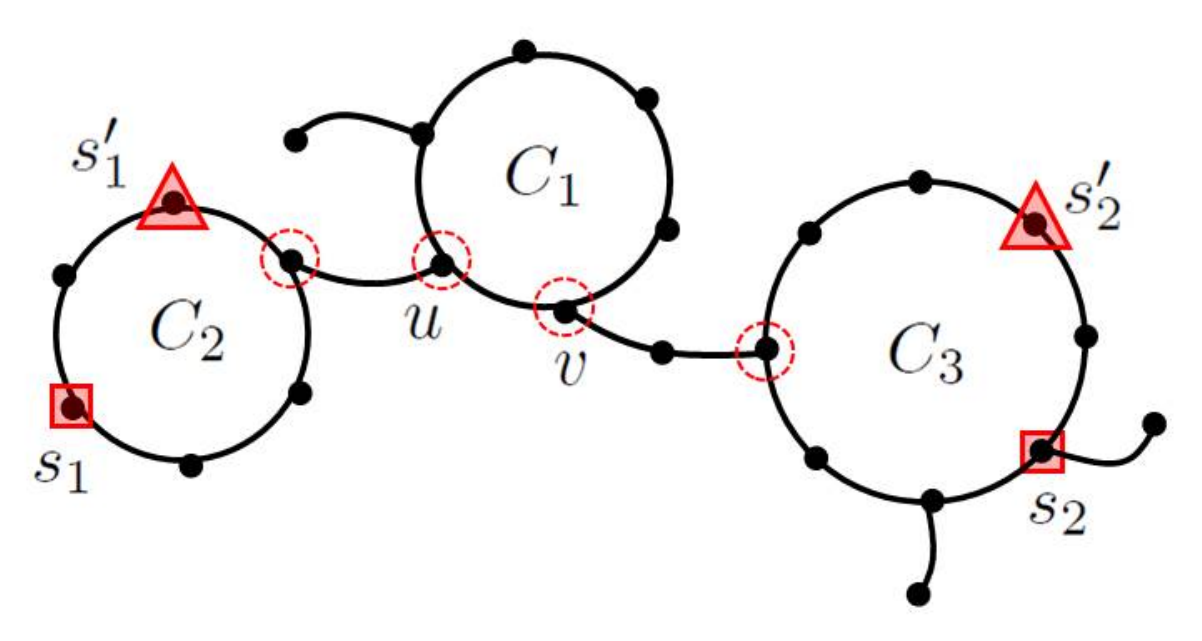}
\end{center}
\caption{A cactus graph with three cycles and two different smallest biactive
branch-resolving sets $S=\{s_{1},s_{2}\}$ and $S^{\prime}=\{s_{1}^{\prime
},s_{2}^{\prime}\}.$ With respect to $S$ the cycle $C_{2}$ contains
configuration $\mathcal{A}$ and $C_{3}$ configurations $\mathcal{B},$
$\mathcal{C}$ and $\mathcal{D}.$ With respect to $S^{\prime}$ cycles $C_{2}$
and $C_{3}$ contain none of the five configurations. The cycle $C_{1}$ has
$b(C_{1})=2,$ so the set of $S$-active vertices $\{u,v\}$ on $C_{1}$ is the
same for all smallest biactive branch-resolving sets and $C_{1}$ contains
configurations $\mathcal{B}$ and $\mathcal{D}$ with respect to all of them,
including $S$ and $S^{\prime}$ shown in the figure. Therefore, in this graph
cycle $C_{1}$ is both $\mathcal{ABC}$- and $\mathcal{ADE}$-positive, and
cycles $C_{2}$ and $C_{3}$ are both $\mathcal{ABC}$- and $\mathcal{ADE}%
$-negative.}%
\label{Fig_avoiding}%
\end{figure}Let $G$ be a cactus graph and $S$ a smallest biactive
branch-resolving set in $G.$ Then
\[
\left\vert S\right\vert =L(G)+B(G),
\]
where $B(G)=\sum_{i=1}^{c}\max\{0,2-b(C_{i})\}.$ If $b(C_{i})\geq2$, then the
set of $S$-active vertices on $C_{i}$ is the same for all smallest biactive
branch-resolving sets $S$. The set of $S$-active vertices may differ only on
cycles $C_{i}$ with $b(C_{i})<2.$ Therefore, such a cycle $C_{i}$ may contain
one of the configurations with respect to one smallest biactive
branch-resolving set, but not with respect to another. This is illustrated by
Figure \ref{Fig_avoiding}.

\begin{definition}
We say that a cycle $C_{i}$ from a cactus graph $G$ is $\mathcal{ABC}%
$\emph{-negative} (resp. $\mathcal{ADE}$\emph{-negative}), if there exists a
smallest biactive branch-resolving set $S$ in $G$ such that $C_{i}$ does not
contain any of the configurations $\mathcal{A},$ $\mathcal{B},$ and
$\mathcal{C}$ (resp. $\mathcal{A},$ $\mathcal{D},$ and $\mathcal{E}$) with
respect to $S.$ Otherwise, we say that $C_{i}$ is $\mathcal{ABC}%
$\emph{-positive} (resp. $\mathcal{ADE}$\emph{-positive}). The number of
$\mathcal{ABC}$-positive (resp. $\mathcal{ADE}$-positive) cycles in $G$ is
denoted by $c_{\mathcal{ABC}}(G)$ (resp. $c_{\mathcal{ADE}}(G)$).
\end{definition}

For two distinct smallest biactive branch-resolving sets $S$, the set of
$S$-active vertices may differ only on cycles with $b(C_{i})\leq1.$ Let
$C_{i}$ and $C_{j}$ be two such cycles in $G$ and notice that the choice of
the vertices included in $S$ from the region of $C_{i}$ is independent of the
choice from $C_{j}.$ Therefore, there exists at least one smallest biactive
branch-resolving set $S$ such that every $\mathcal{ABC}$-negative (resp.
$\mathcal{ADE}$-negative) avoids the three configurations with respect to $S.$
Notice that there may exist more than one such set $S,$ and in that case they
all have the same size, so among them we may choose the one with the smallest
number of vertex-critical (resp. edge-critical) incidencies. Therefore, we say
that a smallest biactive branch-resolving set $S$ is \emph{nice} if every
$\mathcal{ABC}$-negative (resp. $\mathcal{ADE}$-negative) cycle $C_{i}$ does
not contain the three configurations with respect to $S$ and the number of
vertex-critically (resp. edge-critically) incident pairs of cycles with
respect to $S$ is the smallest possible. The niceness of a smallest biactive
branch-resolving set is illustrated by Figure \ref{Fig_optimality}.

\begin{figure}[h]
\begin{center}
\includegraphics[scale=0.8]{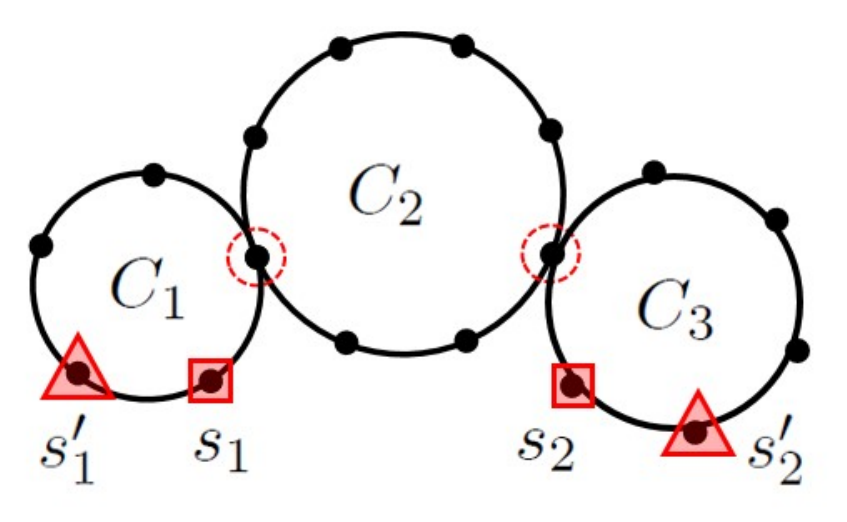}
\end{center}
\caption{A cactus graph with three cycles and two different smallest biactive
branch-resolving sets $S=\{s_{1},s_{2}\}$ and $S^{\prime}=\{s_{1}^{\prime
},s_{2}^{\prime}\}.$ With respect to both $S$ and $S^{\prime}$ all three
cycles do not contain any of the five configurations. The difference is that
with respect to $S$ the cycle $C_{2}$ is both vertex- and edge-critically
incident with both $C_{1}$ and $C_{3}$, and with respect to $S^{\prime}$ the
cycle $C_{2}$ is both vertex- and edge-critically incident with $C_{3}$, but
only edge-critically incident with $C_{1}$. Therefore, the set $S$ is not
nice. Since the critical incidences of $S^{\prime}$ cannot be further avoided
the set $S^{\prime}$ is nice.}%
\label{Fig_optimality}%
\end{figure}

A set $S\subseteq V(G)$ is a \emph{vertex cover} if it contains a least one
end-vertex of every edge in $G.$ The cardinality of a smallest vertex cover in
$G$ is the \emph{vertex cover number} denoted by $\tau(G).$ Now, let $G$ be a
cactus graph and let $S$ be a nice smallest biactive branch-resolving set in
$G.$ We define the \emph{vertex-incident graph} $G_{vi}$ (resp.
\emph{edge-incident graph} $G_{ei}$) as a graph containing a vertex for every
cycle in $G,$ where two vertices are adjacent if the corresponding cycles in
$G$ are $\mathcal{ABC}$-negative and vertex-critically incident (resp.
$\mathcal{ADE}$-negative and edge-critically incident) with respect to $S$.
For example, if we consider the cactus graph $G$ from Figure \ref{Fig_example}%
, then $V(G_{vi})=V(G_{ei})=\{C_{i}:i=1,\ldots,7\},$ where $E(G_{vi}%
)=\{C_{3}C_{4},C_{4}C_{5}\}$ and $E(G_{ei})=\{C_{2}C_{3},C_{3}C_{4},C_{4}%
C_{5}\}.$

We are now in a position to establish the following theorem which gives us the
value of the vertex and the edge metric dimensions in a cactus graph.

\begin{theorem}
\label{Tm_dim}Let $G$ be a cactus graph. Then
\[
\mathrm{dim}(G)=L(G)+B(G)+c_{\mathcal{ABC}}(G)+\tau(G_{vi}),
\]
and
\[
\mathrm{edim}(G)=L(G)+B(G)+c_{\mathcal{ADE}}(G)+\tau(G_{ei}).
\]

\end{theorem}

\begin{proof}
If there is a cycle in $G$ with $b(C)=0,$ then $G$ is a unicyclic graph. For
unicyclic graphs with $b(C)=0$ we have $B(G)=2$ and if the three
configurations $\mathcal{A}$, $\mathcal{B}$, $\mathcal{C}$ (resp.
$\mathcal{A}$, $\mathcal{D}$, $\mathcal{E}$) cannot be avoided by choosing two
vertices into $S,$ then $c_{\mathcal{ABC}}(G)=1$ (resp. $c_{\mathcal{ADE}%
}(G)=1$) and also the third vertex must be introduced to $S,$ so the claim
holds. In all other situations $B(G)$ equals the number of cycles in $G$ with
$b(C_{i})=1.$

Let $S$ be a smallest vertex (resp. edge) metric generator in $G$. Due to
Lemma \ref{Lemma_biactive_branchResolving} the set $S$ must be
branch-resolving. Let $S_{1}\subseteq S$ be a smallest branch-resolving set
contained in $S,$ so $\left\vert S_{1}\right\vert =L(G).$ Since according to
Lemma \ref{Lemma_biactive_branchResolving} the set $S$ must also be biactive,
let $S_{2}\subseteq S\backslash S_{1}$ be a smallest set such that $S_{1}\cup
S_{2}$ is biactive. Obviously, $S_{1}\cap S_{2}=\phi$ and $\left\vert
S_{2}\right\vert =B(G).$

Since $S_{1}\cup S_{2}$ is a smallest biactive branch-resolving set in $G,$ it
follows that every $\mathcal{ABC}$-positive (resp. $\mathcal{ADE}$-positive)
cycle in $G$ contains at least one of the three configurations with respect to
$S_{1}\cup S_{2},$ so according to Theorem \ref{Cor_generatorCharacterization}
the set $S_{1}\cup S_{2}$ is not a vertex (resp. an edge) metric generator in
$G.$ Therefore, each $\mathcal{ABC}$-positive (resp. $\mathcal{ADE}$-positive)
cycle $C_{i}$ must contain a vertex $s_{i}\in S\backslash(S_{1}\cup S_{2}),$
where we may assume that $s_{i}$ is chosen so that it forms a geodesic triple
on $C_{i}$ with vertices from $S_{1}\cup S_{2},$ so according to Observation
\ref{Obs_geodTriple} the cycle $C_{i}$ will not contain any of the
configurations with respect to $S.$ Denote by $S_{3}$ the set of vertices
$s_{i}$ from every $\mathcal{ABC}$-positive (resp. $\mathcal{ADE}$-positive)
cycle in $G.$ Obviously, $S_{3}\subseteq S,$ $S_{1}\cap S_{2}\cap S_{3}=\phi$
and $\left\vert S_{3}\right\vert =c_{\mathcal{ABC}}(G)$ (resp. $\left\vert
S_{3}\right\vert =c_{\mathcal{ADE}}(G)$).

Notice that $S_{1}\cup S_{2}\cup S_{3}$ is a biactive branch-resolving set in
$G$ such that every cycle $C_{i}$ in $G$ does not contain any of the
configurations $\mathcal{A},$ $\mathcal{B},$ $\mathcal{C}$ (resp.
$\mathcal{A},$ $\mathcal{D},$ $\mathcal{E}$) with respect to it. Notice that
$S_{1}\cup S_{2}\cup S_{3}$ still may not be a vertex (resp. an edge) metric
generator, as there may exist vertex-critically (resp. edge-critically)
incident cycles in $G$ with respect to $S_{1}\cup S_{2}\cup S_{3}.$ Since $S$
is a smallest vertex (resp. edge) metric generator, we may assume that $S_{2}$
is chosen so that a smallest biactive branch-resolving set $S_{1}\cup S_{2}$
is nice. Therefore, the graph $G_{vi}$ (resp. $G_{ei}$) contains an edge for
every pair of cycles in $G$ which are $\mathcal{ABC}$-negative and
vertex-critically incident (resp. $\mathcal{ADE}$-negative and edge-critically
incident) with respect to $S_{1}\cup S_{2}.$ Let us denote $S_{4}%
=S\backslash(S_{1}\cup S_{2}\cup S_{3}).$ For each edge $xy$ in $G_{vi}$
(resp. $G_{ei}$) the set $S_{4}$ must contain a vertex from $C_{x}$ or
$C_{y},$ chosen so that it forms a geodesic triple of $S$-active vertices on
$C_{x}$ or $C_{y}$ with other vertices from $S.$ Therefore, $S_{4}$ must
contain at least $\tau(G_{vi})$ (resp. $\tau(G_{ei})$) vertices in order for
$S$ to be a vertex (resp. an edge) metric generator. Since $S$ is a smallest
vertex (resp. edge) metric generator, it must hold $\left\vert S_{4}%
\right\vert =\tau(G_{vi})$ (resp. $\left\vert S_{4}\right\vert =\tau(G_{ei})$).

We have established that $S=S_{1}\cup S_{2}\cup S_{3}\cup S_{4},$ where
$S_{1}\cap S_{2}\cap S_{3}\cap S_{4},$ so $\left\vert S\right\vert =\left\vert
S_{1}\right\vert +\left\vert S_{2}\right\vert +\left\vert S_{3}\right\vert
+\left\vert S_{4}\right\vert .$ Since we also established $\left\vert
S_{1}\right\vert =L(G),$ $\left\vert S_{2}\right\vert =B(G),$ $\left\vert
S_{3}\right\vert =c_{\mathcal{ABC}}(G)$ (resp. $\left\vert S_{3}\right\vert
=c_{\mathcal{ADE}}(G)$) and $\left\vert S_{4}\right\vert =\tau(G_{vi})$ (resp.
$\left\vert S_{4}\right\vert =\tau(G_{ei})$), the proof is finished.
\end{proof}

The formulas for the calculation of metric dimensions from the above theorem
are illustrated by the following examples.

\begin{example}
Let us consider the cactus graph $G$ from Figure \ref{Fig_avoiding}. The set
$S^{\prime}=\{s_{1}^{\prime},s_{2}^{\prime}\}$ is an optimal smallest biactive
branch-resolving set in $G.$ But, since $C_{1}$ is both $\mathcal{ABC}$- and
$\mathcal{ADE}$-positive, the set $S^{\prime}$ is neither a vertex nor an edge
metric generator. Let $s_{3}^{\prime}$ be any vertex from $C_{1}$ which forms
a geodesic triple with two $S^{\prime}$-active vertices on $C_{1}.$ Then the
set $S=\{s_{1}^{\prime},s_{2}^{\prime},s_{3}^{\prime}\}$ is a smallest vertex
(resp. edge) metric generator, so we obtain
\[
\mathrm{dim}(G)=L(G)+B(G)+c_{\mathcal{ABC}}(G)+\tau(G_{vi})=0+2+1+0=3.
\]
and%
\[
\mathrm{edim}(G)=L(G)+B(G)+c_{\mathcal{ADE}}(G)+\tau(G_{ei})=0+2+1+0=3.
\]

\end{example}

\begin{figure}[h]
\begin{center}
\includegraphics[scale=0.9]{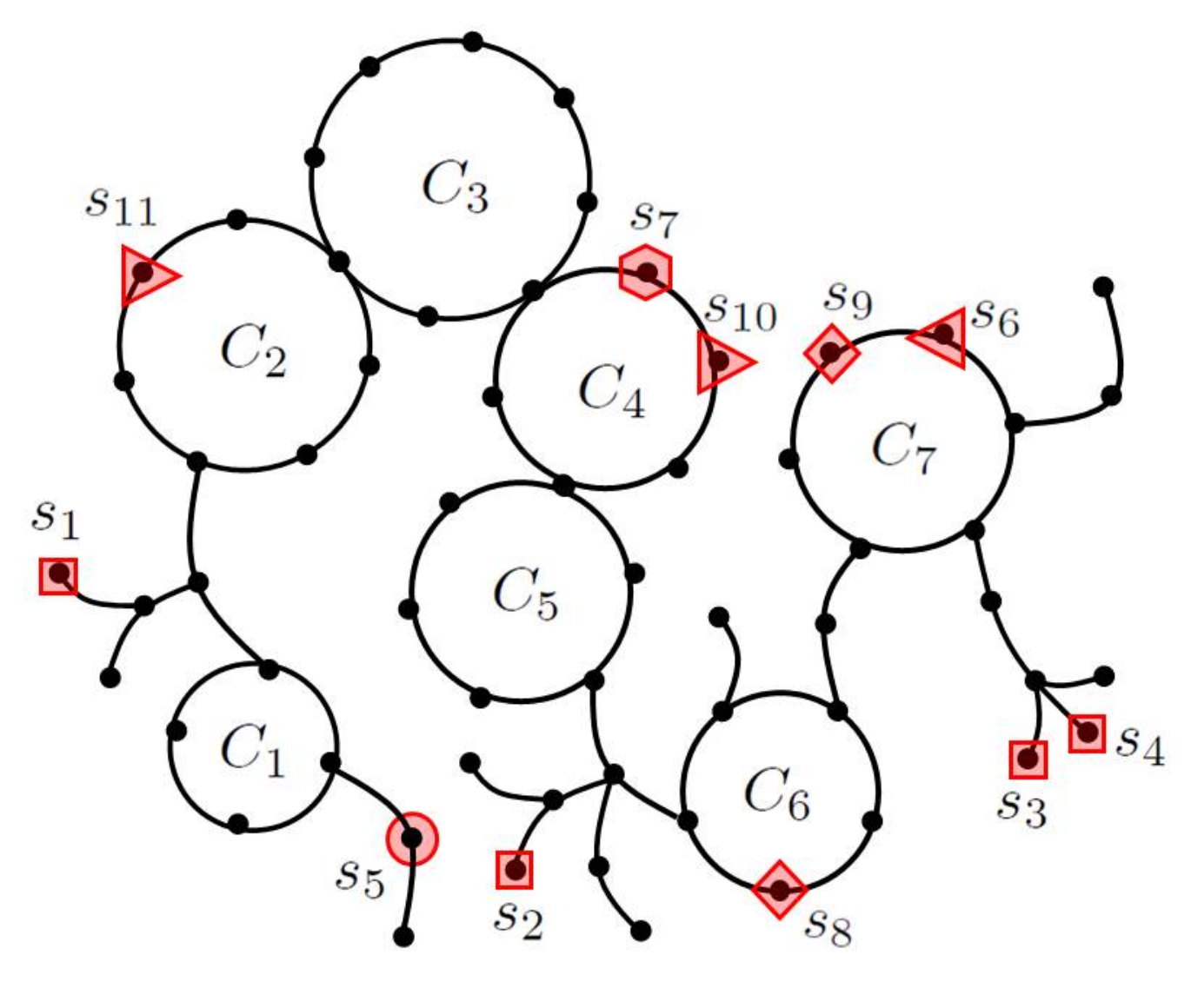}
\end{center}
\caption{A cactus graph $G$ from Example \ref{Example_calc}.}%
\label{Fig_example}%
\end{figure}

Let us now give an example of determining the vertex and the edge metric
dimensions on a cactus graph which is a bit bigger.

\begin{example}
\label{Example_calc}Let $G$ be the cactus graph from Figure \ref{Fig_example}.
The following table gives the choice and the number of vertices for every
expression in the formulas for metric dimensions from Theorem \ref{Tm_dim}
\[%
\begin{tabular}
[c]{|l||l|l|}\hline
& vertices & value\\\hline\hline
$L(G)$ & $s_{1},s_{2},s_{3},s_{4}$ & $4$\\\hline
$B(G)$ & $s_{5}$ & $1$\\\hline
$c_{\mathcal{ABC}}(G)$ & $s_{6}$ & $1$\\\hline
$\tau(G_{vi})$ & $s_{7}$ & $1$\\\hline
$c_{\mathcal{ADE}}(G)$ & $s_{8},s_{9}$ & $2$\\\hline
$\tau(G_{ei})$ & $s_{10},s_{11}$ & $2$\\\hline
\end{tabular}
\ \ \
\]
Therefore, the set $S=\{s_{1},s_{2},s_{3},s_{4},s_{5},s_{6},s_{7}\}$ is a
smallest vertex metric generator, so we obtain
\[
\mathrm{dim}(G)=L(G)+B(G)+c_{\mathcal{ABC}}(G)+\tau(G_{vi})=4+1+1+1=7.
\]
On the other hand, the set $S=\{s_{1},s_{2},s_{3},s_{4},s_{5},s_{8}%
,s_{9},s_{10},s_{11}\}$ is a smallest edge metric generator, so we have%
\[
\mathrm{edim}(G)=L(G)+B(G)+c_{\mathcal{ADE}}(G)+\tau(G_{ei})=4+1+2+2=9.
\]

\end{example}

Notice that $c_{\mathcal{ABC}}(G)\leq c.$ Also, if $\tau(G_{vi})\geq1$ then
$c_{\mathcal{ABC}}(G)+\tau(G_{vi})<c.$ The similar holds for $c_{\mathcal{ADE}%
}(G)$ and $\tau(G_{ei}).$ From this and Theorem \ref{Tm_dim} we immediately
obtain the following result.

\begin{corollary}
\label{Cor_boundB}Let $G$ be a cactus graph with $c$ cycles. Then
$\mathrm{dim}(G)\leq L(G)+B(G)+c$ and $\mathrm{edim}(G)\leq L(G)+B(G)+c.$
\end{corollary}

Further, notice that in a cactus graph with at least two cycles every cycle
has at least one branch-active vertex. Therefore, in such a cactus graph $G,$
we have $B(G)=\sum_{i=1}^{c}\max\{0,2-b(C_{i})\}\leq c$ with equality holding
only if $b(C_{i})=1$ for every cycle $C_{i}$ in $G$. Since $c_{\mathcal{ABC}%
}(G)+\tau(G_{vi})=c$ if and only if $c_{\mathcal{ABC}}(G)=c$ and $\tau
(G_{vi})=0,$ and similarly holds for the edge version of metric dimension,
Theorem \ref{Tm_dim} immediately implies the following simple upper bound on
the vertex and edge metric dimensions of a cactus graph $G$.

\begin{corollary}
Let $G$ be a cactus graph with $c\geq2$ cycles. Then
\[
\mathrm{dim}(G)\leq L(G)+2c
\]
with equality holding if and only if every cycle in $G$ is $\mathcal{ABC}%
$-positive and contains precisely one branch-active vertex.
\end{corollary}

\begin{corollary}
Let $G$ be a cactus graph with $c\geq2$ cycles. Then
\[
\mathrm{edim}(G)\leq L(G)+2c
\]
with equality holding if and only if every cycle in $G$ is $\mathcal{ADE}%
$-positive and contains precisely one branch-active vertex.
\end{corollary}

Notice that the upper bound from the above corollary may not hold for $c=1$,
i.e. for unicyclic graphs, as for the cycle $C$ of unicyclic graph it may hold
that $b(C)=0.$ As for the tightness of these bounds, we have the following proposition.

\begin{proposition}
For every pair of integers $b\geq0$ and $c\geq2,$ there is a cactus graph $G$
with $c$ cycles such that $L(G)=b$ and $\mathrm{dim}(G)=\mathrm{edim}%
(G)=L(G)+2c.$
\end{proposition}

\begin{proof}
For a given pair of integers $b\geq0$ and $c\geq2,$ we construct a cactus
graph $G$ in a following way. Let $G_{0}$ be a graph on $b+2$ vertices, with
one vertex $u$ of degree $b+1$ and all other vertices of degree $1$, i.e.
$G_{0}$ is a star graph. Let $H$ be a graph obtained from the $6$-cycle by
introducing a leaf to it and let $G_{1},\ldots,G_{c}$ be $c$ vertex disjoint
copies of $H$. Denote by $v_{i}$ the only vertex of degree $3$ in $G_{i}.$ Let
$G$ be a graph obtained from $G_{0},G_{1},\ldots,G_{c}$ by connecting them
with an edge $uv_{i}$ for $i=1,\ldots,c$. Obviously, $G$ is a cactus graph
with $c$ cycles and $L(G)=b$. On each of the cycles in $G$ the vertex $v_{i}$
is the only branch-active vertex. If $S\subseteq V(G)$ is a smallest
branch-resolving set in $G$ such that there is a cycle $C_{i}$ in $G$ with
only two $S$-active vertices, then because of the leaf pending on $v_{i}$ the
cycle $C_{i}$ contains either configuration $\mathcal{A}$ if the pair of
$S$-active vertices on $C_{i}$ is an antipodal pair or both configuration
$\mathcal{B}$ and $\mathcal{D}.$ Either way, $S$ is not a vertex nor an edge
metric generator.

On the other hand, the set $S$ consisting of $b$ leaves hanging at $u$ in
$G_{0}$ and a pair of vertices from each $6$-cycle which form a geodesic
triple with $v_{i}$ on the cycle is both a vertex and an edge metric generator
in $G.$ Since $\left\vert S\right\vert =b+2c=L(G)+2c,$ the claims hold.
\end{proof}

\section{An application to zero forcing number}

The results from previous section enable us to solve for cactus graphs a
conjecture posed in literature \cite{Eroh} which involves the vertex metric
dimension, the zero forcing number and the cyclomatic number $c(G)=\left\vert
E(G)\right\vert -\left\vert V(G)\right\vert +1$ (which is sometimes called the
cycle rank number and denoted by $r(G)$) of a graph $G$. Notice that in a
cactus graph $G$ the cyclomatic number $c(G)$ equals the number of cycles in
$G.$ Let us first define the zero forcing number of a graph.

Assuming that every vertex of a graph $G$ is assigned one of two colors, say
black and white, the set of vertices which are initially black is denoted by
$S.$ If there is a black vertex with only one white neighbor, then the
\emph{color-change rule} converts the only white neighbor also to black. This
is one iteration of color-change rule, it can be applied iteratively. A
\emph{zero forcing set} is any set $S\subseteq V(G)$ such that all vertices of
$G$ are colored black after applying the color-change rule finitely many
times. The cardinality of the smallest zero forcing set in a graph $G$ is
called the \emph{zero forcing number} of $G$ and it is denoted by $Z(G).$ In
\cite{Eroh} it was proven that for a unicyclic graph $G$ it holds that
$\mathrm{dim}(G)\leq Z(G)+1$, and it was further conjectured the following.

\begin{conjecture}
\label{Con_zero}For any graph $G$ it holds that $\mathrm{dim}(G)\leq
Z(G)+c(G).$
\end{conjecture}

Moreover, they proved for cacti with even cycles the bound $\mathrm{dim}%
(G)\leq Z(G)+c(G).$ We will use our results to prove that for cacti the
tighter bound from the above conjecture holds.

\begin{proposition}
Let $G$ be a cactus graph. Then $\mathrm{dim}(G)\leq Z(G)+c(G)$ and
$\mathrm{edim}(G)\leq Z(G)+c(G).$
\end{proposition}

\begin{proof}
Due to Corollary \ref{Cor_boundB} it is sufficient to prove that
$L(G)+B(G)\leq Z(G).$ Let $S\subseteq V(G)$ be a zero forcing set in $G.$ Let
us first show that $S$ must be a branch-resolving set. Assume the contrary,
i.e. that $S$ is not a branch-resolving set and let $v\in V(G)$ be a vertex of
degree $\geq3$ with at least two $S$-free threads hanging at $v.$ But then $v$
has at least two white neighbors, one on each of the $S$-free threads hanging
at it, which cannot be colored black by $S,$ so $S$ is not a zero forcing set,
a contradiction.

Let $S_{1}\subseteq S$ be a smallest branch-resolving set contained in $S$ and
let $S_{2}=S\backslash S_{1}.$ Obviously, $\left\vert S_{1}\right\vert =L(G)$
and $S_{1}\cap S_{2}=\phi.$ We now wish to prove that $\left\vert
S_{2}\right\vert \geq B(G).$ If $G$ is a tree, then the claim obviously holds,
so let us assume that $G$ contains at least one cycle. Let $C_{i}$ be a cycle
in $G$ such that $b(C_{i})\leq1.$ If $b(C_{i})=0,$ then $G$ is a unicyclic
graph and $C_{i}$ the only cycle in $G$. Since $b(C_{i})=0,$ we have
$S_{1}=\phi,$ so $S_{2}=S.$ Since a zero forcing set in unicyclic graph must
contain at least two vertices, we obtain $\left\vert S_{2}\right\vert
=\left\vert S\right\vert \geq2=B(G)$ and the claim is proven.

Assume now that for every cycle $C_{i}$ with $b(C_{i})\leq1$ it holds that
$b(C_{i})=1.$ Let $v$ be the branch-active vertex on such a cycle $C_{i}$ and
notice that $S_{1}$ can turn only $v$ black on $C_{i}.$ Therefore, in order
for $S$ to be a zero forcing set it follows that $S_{2}$ must contain a vertex
from every such cycle, i.e. $\left\vert S_{2}\right\vert \geq B(G).$
Therefore, $\left\vert S\right\vert =\left\vert S_{1}\right\vert +\left\vert
S_{2}\right\vert \geq L(G)+B(G).$
\end{proof}

The above proposition, besides proving for cacti the cycle rank conjecture
which was posed for $\mathrm{dim}(G),$ also gives a similar result for
$\mathrm{edim}(G).$ So, this motivates us to pose for $\mathrm{edim}(G)$ the
counterpart conjecture of Conjecture \ref{Con_zero}.

\section{Concluding remarks}

In \cite{SedSkreBounds} it was established that for a unicyclic graph $G$ both
vertex and edge metric dimensions are equal to $L(G)+\max\{2-b(G),0\}$ or
$L(G)+\max\{2-b(G),0\}+1.$ In \cite{SedSkreUnicyclic} a characterization under
which both of the dimensions take one of the two possible values was further
established. In this paper we extend the result to cactus graphs where a
similar characterization must hold for every cycle in a graph, and also the
additional characterization for the connection of two cycles must be
introduced. This result enabled us to prove the cycle rank conjecture for
cactus graphs.

Moreover, the results of this paper enabled us to establish a simple upper
bound on the value of the vertex and the edge metric dimension of a cactus
graph $G$ with $c$ cycles%
\[
\mathrm{dim}(G)\leq L(G)+2c\quad\hbox{ and }\quad\mathrm{edim}(G)\leq
L(G)+2c.
\]
Since the number of cycles can be generalized to all graphs as the cyclomatic
number $c(G)=\left\vert E(G)\right\vert -\left\vert V(G)\right\vert +1,$ we
conjecture that the analogous bounds hold in general.

\begin{conjecture}
Let $G$ be a connected graph. Then, $\mathrm{dim}(G)\leq L(G)+2c(G).$
\end{conjecture}

\begin{conjecture}
Let $G$ be a connected graph. Then, $\mathrm{edim}(G)\leq L(G)+2c(G).$
\end{conjecture}

In \cite{SedSkrekMixed} it was shown that the inequality $\mathrm{mdim}%
(G)<2c(G)$ holds for $3$-connected graphs. Since $\mathrm{dim}(G)\leq
\mathrm{mdim}(G)$ and $\mathrm{edim}(G)\leq\mathrm{mdim}(G),$ the previous two
conjectures obviously hold for $3$-connected graphs.

Also, motivated by the bound on edge metric dimension of cacti involving zero
forcing number, we state the following conjecture for general graphs, as a
counterpart of Conjecture \ref{Con_zero}.

\begin{conjecture}
Let $G$ be a connected graph. Then, $\mathrm{edim}(G)\leq Z(G)+c(G).$
\end{conjecture}

\bigskip

\bigskip\noindent\textbf{Acknowledgments.}~~Both authors acknowledge partial
support of the Slovenian research agency ARRS program\ P1-0383 and ARRS
projects J1-1692 and J1-8130. The first author also the support of Project
KK.01.1.1.02.0027, a project co-financed by the Croatian Government and the
European Union through the European Regional Development Fund - the
Competitiveness and Cohesion Operational Programme.

\end{document}